\documentclass[12pt]{article}
\nonstopmode
\RequirePackage[colorlinks,citecolor=blue,urlcolor=blue,linkcolor=blue]{hyperref}
\hypersetup{
colorlinks = true,
citecolor=blue,
urlcolor=blue,
linkcolor=blue,
pdfpagemode = UseNone
}
\usepackage{graphicx,xspace,colortbl}
\usepackage{amsmath,amsthm,amsfonts,amssymb}
\usepackage{color}
\usepackage{fancybox}
\usepackage{epsfig}
\usepackage{subfig}
\usepackage{pdfsync}
    \hypersetup{
colorlinks = true,
citecolor=blue,
urlcolor=blue,
linkcolor=blue,
pdfauthor = {Alexey Kuznetsov},
pdfkeywords = {free regular measure, complete Bernstein function, Bondesson class, subordinator, 
 Kendall identity, Laplace transform, Cauchy transform, Bessel functions, Lambert W-function},
pdftitle = {On free regular and Bondesson convolution semigroups},
pdfpagemode = UseNone
}
    \def\qed{\hfill$\sqcap\kern-8.0pt\hbox{$\sqcup$}$\\}
    \def\beq{\begin{eqnarray}}
    \def\eeq{\end{eqnarray}}
    \def\beqq{\begin{eqnarray*}}
    \def\eeqq{\end{eqnarray*}}

    \def\re{\textnormal {Re}}
    \def\im{\textnormal {Im}}
    \def\p{{\mathbb P}}
    \def\e{{\mathbb E}}
    \def\r{{\mathbb R}}
    \def\c{{\mathbb C}}
    	
    \def\d{{\textnormal d}}
    \def\i{{\textnormal i}}

\newtheorem{theorem}{Theorem}

\newtheorem{proposition}{Proposition}
\newtheorem{corollary}{Corollary}
\theoremstyle{definition}
\newtheorem{definition}{Definition}

\newtheorem{remark}{Remark}

\title{On free regular and Bondesson convolution semigroups}

\author{
A. Kuznetsov
\footnote{Department of Mathematics and Statistics, 
York University, 
4700 Keele Street, 
Toronto, Ontario, 
M3J 1P3, Canada. Email: kuznetsov@mathstat.yorku.ca}
}
\date{\footnotesize This version: \today}

\begin{document}

\maketitle

\begin{abstract}
\bigskip
Free regular convolution semigroups describe the distribution of free subordinators, while Bondesson class convolution semigroups correspond to classical subordinators with completely monotone L\'evy density. We show that these two classes of convolution semigroups are in bijection with the class of complete Bernstein functions and we establish an integral identity linking the two semigroups. We provide several explicit examples that illustrate this result. 
\end{abstract}

{\vskip 0.5cm}
{\vskip 0.5cm}

\section{Introduction}\label{section_Intro}

In this paper we study the distributions of classical subordinators, their free analogues and connections between the two. Let us first recall some key concepts. 
One-dimensional stochastic process $Y=\{Y_t\}_{t\ge 0}$ is called a subordinator 
if it is an increasing L\'evy process (see \cite{Bertoin}). Subordinators give rise to convolution semigroups $\nu_t(\d y)=\p(Y_t \in \d y)$ supported on $[0,\infty)$, which satisfy $\nu_t * \nu_s=\nu_{(t+s)}$ for $t,s\ge 0$ (here $*$ denotes the classical convolution of measures). We will denote $\nu^{*t}=\nu_t$ from now on. 

Free subordinators are non-commutative analogues of classical subordinators. A free subordinator is a free additive
 L\'evy process  of the first kind (see \cite{Biane_1998}) such that the distribution of $X_t-X_s$ is supported on $[0,\infty)$ for all $0<s<t$. Denoting by $\mu_t$ the distribution of $X_t$ we obtain a $\boxplus$-convolution semigroup supported on $[0,\infty)$, which satisfies 
$\mu_t \boxplus \mu_s=\mu_{t+s}$ (here $\boxplus$ denotes the free convolution of probability measures). Similarly to the classical case, we will denote $\mu^{\boxplus t}=\mu_t$.
If $X$ is a free subordinator then the distribution of $X_1$ is called a free regular measure \cite{Arizmendi_2013,Sakuma} and we will call $\mu^{\boxplus t}$ a free regular convolution semigroup.

Classical convolution semigroups on $[0,\infty)$ are typically described by their Laplace exponent: a function 
$f: [0,\infty) \mapsto [0,\infty)$ such that 
\begin{equation}\label{def_Laplace_exponent}
\e[e^{-z Y_t}]=\int_{[0,\infty)} e^{-zy } \nu^{*t}(\d y)=e^{-t f(z)}, \;\;\; z\ge 0. 
\end{equation}
It is well known that for \eqref{def_Laplace_exponent} to hold, $f$ must be a Bernstein function \cite{Bernstein_book}. Similarly, free regular convolution semigroups are described by a function $\phi$ such that 
\begin{equation}\label{def_free_regular_phi}
\phi_{\mu^{\boxplus t}}(z)=t\phi(z), \;\;\; z\le 0.
\end{equation}
Here $\phi_{\mu}$ is the Voiculescu transform of a probability measure $\mu$, which we define below in Section \ref{section_Bernstein_and_free_stuff}. 

The key observation that inspired this paper is the following: if $\phi$ is a function
defining the free regular semigroup $\mu^{\boxplus t}$ via 
\eqref{def_free_regular_phi}, then the function $f(z)=\phi(-z)$ is a Bernstein function and thus 
it gives rise to a classical convolution semigroup $\nu^{*t}$ via \eqref{def_Laplace_exponent}. This observation 
leads to the following question: are there any interesting connections between the two semigroups? Our main result in this paper states that there indeed exists a rather  surprising identity that links the $*$-convolution semigroup $\nu^{*t}(\d y)$ with the
corresponding $\boxplus$-convolution semigroup $\mu^{\boxplus t}(\d x)$.

This paper is organized as follows. In the next three sections we introduce some background material about complete Bernstein functions, Bondesson class of distributions and free regular measures, which will be required to state and prove our main result in Section \ref{section_Results}. In Section \ref{section_Examples} we provide four explicit examples that illustrate our main result.

\subsection{Complete Bernstein functions}\label{section_Bernstein}

We recall that a function $f : [0,\infty) \mapsto [0,\infty)$  is said to be a {\it Bernstein function} \cite{Bernstein_book}  if 
\begin{equation}\label{def_Bernstein}
f(z)=a+bz+\int_0^{\infty} (1-e^{-z x}) \Pi(\d x),
\end{equation}
for $a,b\ge 0$ and a positive measure $\Pi$ supported on $(0,\infty)$ such that $\int_{0}^{\infty} \min(1,x) \Pi(\d x)< \infty$.  
The constant $b$ is called the linear drift coefficient, whereas the measure $\Pi$ is called the L\'evy measure. 
A function $g: (0,\infty) \mapsto (0,\infty)$ is called {\it completely monotone} if $(-1)^n g^{(n)}(x)\ge 0$ for all $n=0,1,2,\dots$. The well-known Bernstein's Theorem tells us that a function $g$ is completely monotone if and only if 
$$
g(x)=\int_{[0,\infty)} e^{-ux} \sigma(\d u),
$$ 
for some measure $\sigma$. 

\begin{definition}
A Bernstein function $f : [0,\infty) \mapsto [0,\infty)$ is said to be a {\it complete Bernstein function} if its L\'evy measure 
$\Pi(\d x)$ in \eqref{def_Bernstein} has a density $\pi(x)$ (with respect to Lebesgue measure) that is completely monotone. 
\end{definition}
Following \cite{Bernstein_book}, we will denote by ${\mathcal{CBF}}$ the class of all complete Bernstein functions. 
 The next result  provides two equivalent characterizations of complete Bernstein functions. The proof can be found in \cite{Bernstein_book}[Chapter 6]. 

\begin{theorem}\label{theorem_CBF}
Let $f$ be a real-valued function defined on $(0,\infty)$. Then the following are equivalent:
\begin{itemize}
\item[(i)] $f\in {\mathcal{CBF}}$.
\item[(ii)] $f$ has an analytic continuation to ${\mathbb H}^{\uparrow}:=\{z\in \c: \im(z)>0\}$ such and $f(z) \in {\mathbb H}^{\uparrow}$ for 
$z \in {\mathbb H}^{\uparrow}$ and $f(z)\ge 0$ for $z >0$. 
\item[(iii)] $f$ has an analytic continuation to ${\mathbb C}\setminus (-\infty,0]$ and for all $z$ in this domain we have 
\begin{align}\label{CBF_integral_representation}
f(z)&=a+bz+\int_{(0,\infty)} \frac{zx-1}{z+x} \rho(\d x)\\
\nonumber
&=a+bz+\int_{(0,\infty)} \Big( \frac{x}{1+x^2} - \frac{1}{z+x} \Big)(1+x^2) \rho(\d x),
\end{align}
where  $\rho$ is a measure on $(0,\infty)$ satisfying $\int_{(0,\infty)} \max(x^{-1},1) \rho(\d x) < \infty$  and $a,b,\ge 0$ are constants with $a\ge \int_{(0,\infty)} x^{-1} \rho(\d x)$. 
\end{itemize}
\end{theorem}

We denote by ${\mathcal{CBF}}^{\,\flat}$ the class of complete Bernstein functions $f$ such that $b=0$, where $b$ is the constant appearing in  
\eqref{CBF_integral_representation}. It can be shown that the constant $b$ in \eqref{CBF_integral_representation} satisfies $b=\lim\limits_{z\to +\infty} f(z)/z$, thus an equivalent definition of the class ${\mathcal{CBF}}^{\,\flat}$ would be
$$
{\mathcal{CBF}}^{\,\flat}:=\{f \in {\mathcal{CBF}} \; : \; \lim\limits_{z\to +\infty} f(z)/z=0 \}. 
$$

\subsection{Bondesson class}\label{section_Bernstein_and_BO}

The following class of sub-probability measures will be important in what follows:
\begin{definition}\label{def_BO_class}
A sub-probability measure $\nu$ on $[0,\infty)$ belongs to the {\it Bondesson class} if 
\begin{equation}\label{eqn_def_BO_class}
\int_{[0,\infty)} e^{-z y} \nu(\d y)=e^{-f(z)}, \;\;\; z>0,
\end{equation}
for some $f \in {\mathcal{CBF}}$. 
\end{definition}

Following \cite{Bernstein_book}, we will denote the Bondesson class by ${\textrm{BO}}$. Clearly, any $\nu \in {\textrm{BO}}$ is infinitely divisible, thus we can also associate to any $f\in {\mathcal{CBF}}$  a $*$-convolution semigroup $\nu^{*t}(\d y)=\nu^{*t}(f;\d y) \in {\textrm{BO}}$ via 
\begin{equation}\label{def_*_semigroup}
\int_{[0,\infty)} e^{-zy} \nu^{*t}(\d y)=e^{-t f(z)}, \;\;\; z>0. 
\end{equation}

The reader can find more information about the Bondesson class in \cite{Bernstein_book}[Chapter 9]. Here we will only mention the following probabilistic description of this class: ${\textrm{BO}}$ is the smallest class of sub-probability measures on $[0,\infty)$ that contains all mixtures of exponential distributions and that is closed under convolutions and vague limits (see \cite{Bernstein_book}[Theorem 9.7]). 

We will denote by ${\textrm{BO}}^{\flat}$ the subclass of the Bondesson class that is obtained by taking $f\in {\mathcal{ CBF}}^{\,\flat}$ in 
\eqref{def_BO_class}. Using the well-known connection between the left extremity of the support of the measure and the asymptotics of its Laplace transform one can give the following equivalent definition: 
$$
{\textrm{BO}}^{\flat}=\{\nu \in {\textrm{BO}} \; : \; \nu[0,\epsilon]>0 \; {\textnormal{ for every }} \; \epsilon>0\}.  
$$
In other words, ${\textrm{BO}}^{\flat}$ is the class of measures in ${\textrm{BO}}$ with the left extremity of support equal to zero. It is clear  that any measure in ${\textrm{BO}}$ can be obtained as the pushforward of some measure in ${\textrm{BO}}^{\flat}$ with respect to the function $x\mapsto x+b$.

\subsection{Free regular measures}\label{section_Bernstein_and_free_stuff}

Let us first define free convolution of probability measures. 
We recall that  ${\mathbb H}^{\uparrow}$ denotes the open upper complex half-plane.
For a probability measure $\mu$ on $\r$ we define its Cauchy transform
\beq
G_{\mu}(z):=\int_{\r} \frac{\mu(\d x)}{z-x}, \;\;\; z\in {\mathbb H}^{\uparrow},
\eeq
and we denote $F_{\mu}(z):=1/G_{\mu}(z)$. Note that 
$F_{\mu}$ maps ${\mathbb H}^{\uparrow}$ into itself. 
It is known  \cite{Bercovici_1993} that there exist $\alpha,M>0$ such that 
$F_{\mu}$ has a right compositional inverse $F^{-1}_{\mu}$ defined in the region
$$
\Gamma_{\alpha,M}:=\{z\in {\mathbb C}: |\re(z)|<\alpha \im(z), \im(z)>M\}.
$$ 
The {\it Voiculescu transform} of $\mu$ is defined by $\phi_{\mu}(z)=F^{-1}_{\mu}(z)-z$, for $z \in \Gamma_{\alpha,M}$.

The free additive convolution of two probability measures $\mu_1$, $\mu_2$ on $\r$ is the measure
$\chi=\mu_1 \boxplus \mu_2$ such that   
$\phi_{\chi}(z)=\phi_{\mu_1}(z)+\phi_{\mu_2}(z)$ on some region $\Gamma_{\alpha,M}$ where both 
$\phi_{\mu_1}(z)$ and $\phi_{\mu_2}(z)$ are defined.

We say that a probability measure $\mu$ on $\r$ is {\it freely infinitely divisible} (or $\boxplus${\it-infinitely divisible}) if for all $n$ there exists a probability measure $\mu_n$ such that 
$\mu=\mu_n^{\boxplus n}$ (meaning $\mu$ is the $n$-fold $\boxplus$-convolution of $\mu_n$ with itself). It is known that $\mu$ is freely infinitely divisible if and only if $\phi_{\mu}$ can be extended to an analytic function in ${\mathbb H}^{\uparrow}$, where it has the following integral representation 
\begin{equation}\label{phi_mu_representation}
\phi_{\mu}(z)=\gamma_{\mu} + \int_{\r} \frac{1+xz}{z-x} \tau_{\mu}(\d x),
\end{equation}
for $\gamma_{\mu}\in \r$ and a finite non-negative measure $\tau_{\mu}$.  

Next we introduce free regular measures: these are free analogues of distribution of a subordinator. 
\begin{definition}\label{definition_free_regular_measures}
A probability measure $\mu$ is called {\it free regular} if it is $\boxplus$-infinitely divisible and $\mu^{\boxplus t}$ is supported on $[0,\infty)$ for all $t>0$.   
 \end{definition}
 
The class of free regular measures will be denoted by $I^{\boxplus}_{r+}$.  Free regular distributions were introduced in \cite{Perez-Abreu2012} and studied later in  \cite{Arizmendi_2013,Sakuma}. Note that free regular measures are defined differently in \cite{Arizmendi_2013}, but our definition \ref{definition_free_regular_measures} is an equivalent one (see \cite{Arizmendi_2013}[Theorem 4.2]).

The following result was essentially established in \cite{Arizmendi_2013}, though it was not stated explicitly in this way. 
\begin{proposition}\label{thm_free_regular}
A probability measure $\mu$ is free regular if and only if the function $z\in [0,\infty) \mapsto \phi_{\mu}(-z)$ belongs to the class ${\mathcal{CBF}}^{\,\flat}$.  
\end{proposition}

To prove this result, we note that Theorem 4.2 in \cite{Arizmendi_2013} tells us that $\mu \in I^{\boxplus}_{r+}$ 
if and only if $\mu$ 
is $\boxplus$-infinitely divisible,  $\phi_{\mu}(-0)\ge 0$ and $\tau_{\mu}(-\infty,0)=0$, where $\tau_{\mu}$ is the measure 
in the representation \eqref{phi_mu_representation}. Thus 
\begin{equation}\label{phim_minus_z}
\phi_{\mu}(-z)=\gamma_{\mu} + \int_{[0,\infty)} \frac{xz-1}{z+x} \tau_{\mu}(\d x)
=\gamma_{\mu} + \int_{[0,\infty)} \Big( \frac{x}{1+x^2} - \frac{1}{z+x} \Big)(1+x^2)\tau_{\mu}(\d x),
\end{equation}
so it is clear that $\phi_{\mu}(-z)$ maps ${\mathbb H}^{\uparrow}$ into ${\mathbb H}^{\uparrow}$. 
It is also clear from \eqref{phim_minus_z} that $\phi_{\mu}(-z)$ is increasing 
on $(0,\infty)$, and since $\phi_{\mu}(-0)\ge 0$ we conclude that $\phi_{\mu}(-z)>0$ for $z>0$. 
Theorem \ref{theorem_CBF}(ii) tells us that $\phi_{\mu}(-z)$ must be a complete Bernstein function, and the fact that 
the linear drift coefficient is zero is obvious from \eqref{phi_mu_representation}. The other direction (showing that for any 
$f \in {\mathcal{CBF}}^{\,\flat}$ there exists a free regular measure  $\mu$ such that $\phi_{\mu}(-z)=f(z)$ for $z>0$) is left to the reader. 

Proposition \ref{thm_free_regular} tells us that with any $f\in {\mathcal{CBF}}^{\,\flat}$ we can associate a $\boxplus$-convolution semigroup $\mu^{\boxplus t}(\d x)=\mu^{\boxplus t}(f; \d x) \in I^{\boxplus}_{r+}$ via the Voiculescu transform
\begin{equation}\label{def_free_semigroup}
\phi_{\mu^{\boxplus t}}(z)=t f(-z), \;\;\;  z<0,
\end{equation}
and, conversely, any $\boxplus$-convolution semigroup in $I^{\boxplus}_{r+}$ can be constructed in such a way.

\begin{remark}
we remind the reader that the measure $\mu$ can be recovered from the Cauchy transform by {\it Stieltjes inversion formula}:
\begin{equation}\label{Stieltjes_inversion}
\mu(u,v)=\lim\limits_{y\to 0^+} -\frac{1}{\pi} \int_u^v \im(G_{\mu}(x+\i y))\d x,
\end{equation}
which is valid for all continuity points $u,v$ of the measure $\mu$. This fact will be used in Section \ref{section_Examples}. 
\end{remark}

\section{Results}\label{section_Results}

As we discussed in Sections \ref{section_Bernstein_and_BO} and \ref{section_Bernstein_and_free_stuff}, the class 
$I^{\boxplus}_{r+}$ of free regular $\boxplus$-convolution semigroups is in bijection with the 
class ${\mathcal{CBF}}^{\,\flat}$ of complete Bernstein functions with zero linear drift, 
and the same is true for the class of $*$-convolution semigroups in ${\textrm{BO}}^{\flat}$. The next theorem is our main result: it establishes an identity between the corresponding semigroups in $I^{\boxplus}_{r+}$ and ${\textrm{BO}}^{\flat}$.

\begin{theorem}\label{theorem_main}
Fix a function $f \in {\mathcal{CBF}}^{\,\flat}$ and define a $*$-convolution semigroup $\nu^{*t}(\d x)=\nu^{*t}(f;\d x)\in {\textrm{BO}}^{\,\flat}$ via \eqref{def_*_semigroup} and a $\boxplus$-convolution semigroup 
 $\mu^{\boxplus t}(\d x)=\mu^{\boxplus t}(f;\d x) \in I^{\boxplus}_{r+}$ via \eqref{def_free_semigroup}. 
Then for all $t> 0$ and $w> 0$ we have 
\begin{equation}\label{eqn_thm_main}
\int_{[0,\infty)} e^{-w x} \mu^{\boxplus t}(\d x)=w^{-1}\int_{0}^w \nu^{*wt}[0,y]\d y. 
\end{equation}
\end{theorem}

\begin{proof}
Without loss of generality, it is enough to consider the case when $f(0+)=0$.  
This is true since for any $f \in {\mathcal{CBF}}^{\,\flat}$ and $\kappa \ge 0$ we have 
 \begin{align*}
 \nu^{*t}(\kappa+f;\d y)&=e^{-\kappa t} \nu^{*t}( f;\d y), \qquad\qquad \;\;\;\;\, \;\;\; t,y\ge 0, \\
 \mu^{\boxplus t}(\kappa + f;\d x)&={\mathbf 1}_{\{x-\kappa t\ge 0\}}  \mu^{\boxplus t}( f;\d (x-\kappa t)), \;\;\; t, x \ge 0. 
 \end{align*}
 These results can be easily derived from the definitions of the semigroups $\nu^{*t}$ and $\mu^{\boxplus t}$ in Sections 
 \ref{section_Bernstein_and_BO} and \ref{section_Bernstein_and_free_stuff} and we leave all the details of this derivation to the reader.

Let us now consider a function  $f \in {\mathcal{CBF}}^{\,\flat}$ such that $f(0+)=0$. Associated to this function we have a subordinator 
$Y=\{Y_t\}_{t\ge 0}$ such that $\e[\exp(- z Y_t)]=\exp(-t f(z))$. In particular, we have 
$\p(Y_t \in \d y)=\nu^{*t}(\d y)$ for $t,y\ge 0$. For each $y\ge 0$ we define $\tau_y=\inf\{ t \ge 0 : t-Y_t \ge y\}$, with the understanding that $\inf \emptyset=+\infty$. The process $\tau=\{\tau_y\}_{y=0}$ is also a (possibly killed) subordinator (see \cite{Kyprianou}[Corollary 3.14]), thus there exists a Bernstein function $\psi$ such that $\psi(z)=-\ln \e[\exp(-z \tau_1)]$. It is known that $\psi$ satisfies the functional equation 
\begin{equation}\label{psi_functional_eqn}
\psi(w)-f(\psi(w))=w, \;\;\; w>0, 
\end{equation}
see \cite{Kyprianou}[Theorem 3.12]. Moreover, the distribution of $\tau$ is related to the distribution of $Y$ via Kendall's identity 
\begin{equation}\label{Kendalls_identity}
\p(\tau_y \in \d s)\d y=\frac{y}{s} \p(s-Y_s \in \d y) \d s,
\end{equation}
see  \cite{ECP1038}, \cite{Burridge}, \cite{Kendall} or  \cite{Kyprianou}[exercise 6.10].

Our next goal is to show that $\psi(z)=-F_{\mu}(-z)$ for $z>0$, where $F_{\mu}(z)=1/G_{\mu}(z)$ and $\mu=\mu^{\boxplus 1}(f;\d x)$. Note that we have $\phi_{\mu}(z)=f(-z)$ for $z<0$ and the function $z \in (-\infty,0) \mapsto F_{\mu}^{-1}(z)=z+\phi_{\mu}(z)$ is the right-compositional inverse of $F_{\mu}$
(see the discussion in Section \ref{section_Bernstein_and_free_stuff}). 
Thus, for $z>0$ we obtain $F_{\mu}^{-1}(-z)=-z+\phi_{\mu}(-z)=-z+f(z)$ and 
\begin{equation}\label{inv_eqn1}
-z=F_{\mu}(F_{\mu}^{-1}(-z))=F_{\mu}(-z+f(z)). 
\end{equation}
At the same time, from \eqref{psi_functional_eqn} we conclude that 
\begin{equation}\label{inv_eqn2}
-z=-\psi(z-f(z))  
\end{equation}
for all $z>0$ such that $z-f(z)>0$. Comparing \eqref{inv_eqn1} and \eqref{inv_eqn2}, we see that 
$\psi(z)=-F_{\mu}(-z)$ for all $z>0$ such that $z-f(z)>0$, and by analytic continuation this must hold for all $z\in {\mathbb C}\setminus(-\infty,0]$. 

Next, using Fubini's Theorem, we obtain 
\begin{align}\label{eqn_proof1}
\int_0^{\infty} e^{-z s} \bigg[ \int_{[0,\infty)} e^{-s x} \mu(\d x) \bigg] \d s&
\nonumber =
\int_{[0,\infty)} \frac{\mu(\d x)}{z+x}\\ & =
-G_{\mu}(-z)=\frac{1}{-F_{\mu}(-z)}=\frac{1}{\psi(z)},   
\end{align}
for all $z>0$. 
Note that $\tau_y$ is a subordinator with unit linear drift, thus according to \cite{Bertoin}[Proposition 1.7] the renewal measure 
$$
U(\d s)=\int_0^{\infty} \p(\tau_y \in \d s) \d y
$$
is absolutely continuous with respect to Lebesgue measure, so that $U(\d s)=u(s)\d s$. It is also well known that the renewal measure satisfies 
\begin{equation}\label{eqn_proof2}
\int_{0}^{\infty} e^{-z s} U(\d s)=\int_{0}^{\infty} e^{-z s} u(s)\d s=\frac{1}{\psi(z)}, \;\;\; z>0. 
\end{equation}
From \eqref{eqn_proof1}, \eqref{eqn_proof2} and uniqueness of the Laplace transform we conclude that 
\begin{equation}\label{eqn_proof3}
\int_{[0,\infty)} e^{-s x} \mu(\d x) = u(s), \;\;\; s\ge 0. 
\end{equation}
We use Kendall's identity \eqref{Kendalls_identity} in 
order to find the density of the renewal measure of the subordinator $\tau$:   
\begin{align}\label{finding_renewal_density}
u(s)\d s=\int_{[0,\infty)} \p(\tau_y \in \d s)\d y &= \bigg[\int_{[0,\infty)} \frac{y}{s} \p(s-Y_s \in \d y) \bigg] \d s
=\bigg[\int_{[0,s)} \frac{s-y}{s} \p(Y_s \in \d y) \bigg]\d s
\\ \nonumber & = \frac{1}{s} 
\bigg[\int_0^s \p(Y_s \le  y)\d y \bigg]\d s=\frac{1}{s} \bigg[\int_0^s \nu^{*s}([0,y]) \d y \bigg] \d s,
\end{align}
where the fourth step is justified by integration by parts. Combining \eqref{eqn_proof3} and 
\eqref{finding_renewal_density} we obtain 
$$
\int_{[0,\infty)} e^{-s x} \mu^{\boxplus 1} (f;\d x) = \frac{1}{s} \int_0^s \nu^{*s}(f;[0,y]) \d y , \;\;\; s>0, 
$$
which is equivalent to \eqref{eqn_thm_main} with $t=1$. To obtain the general case one only needs to replace $f$ by $tf$ and use the fact that 
$$
\mu^{\boxplus 1} (tf;\d x)=\mu^{\boxplus t} (f;\d x) \;\; \textnormal{and} \;\; \nu^{*s}(tf;\d y)=\nu^{*st}(f;\d y). 
$$
\end{proof}

\begin{remark}
While proving Theorem \ref{theorem_main} we have established the following characterization of free regular measures: 
a probability measure $\mu$ is free regular if and only if its Laplace transform is equal to the density of the renewal measure of the  subordinator $\tau=\{\tau_y\}_{y\ge 0}$ obtained from a subordinator $Y=\{Y_t\}_{t\ge 0}$ in the Bondesson class ${\textrm{BO}}^{\flat}$ via 
$$
\tau_y=\inf\{ t \ge 0 : t-Y_t \ge y\}. 
$$. 
\end{remark}

The following corollary of Theorem \ref{theorem_main} shows that the cumulative distribution function of any measure in the Bondesson class ${\textrm{BO}}^{\flat}$ can be expressed in terms of the corresponding $\boxplus$-convolution semigroup. This result is easily derived 
 by rescaling  $t \mapsto t/w$ and taking derivative $\d/\d w$ of both sides of identity 
 \eqref{eqn_thm_main}. 

\begin{corollary}\label{corollary1} Let $\nu  \in {\textrm{BO}}^{\,\flat}$ and let $f \in {\mathcal{CBF}}^{\,\flat}$ be the corresponding Laplace exponent, as defined in \eqref{eqn_def_BO_class}.  
Then  
\begin{equation}\label{eqn_corollary_main}
\nu^{*t}[0,w]=\frac{\d}{\d w} \Bigg[ w \int_{[0,\infty)} e^{-xw} \mu^{\boxplus \tfrac{t}{w}} (f;\d x) \Bigg], \;\;\; t,w>0.  
\end{equation}
\end{corollary}

\section{Examples}\label{section_Examples}

In this section we consider several examples where computations can be carried out explicitly in terms of elementary or special functions.

\subsection*{Example 1}

First we take $f(z)=z^{1-\alpha}$, where $\alpha \in (0,1)$.  It is clear that $f\in {\mathcal{CBF}}$ (see Theorem \ref{theorem_CBF}(ii)). Moreover, we have $f\in {\mathcal{CBF}}^{\,\flat}$, 
since $f(z)/z\to 0$ as $z\to +\infty$. The $*$-convolution semigroup $\nu^*(f;\d y)$ corresponds to the stable subordinator with parameter $\tilde\alpha=1-\alpha$ while the  $\boxplus$-convolution semigroup $\mu^{\boxplus t}(f;\d x)$ comes from the free stable measures with parameters $\alpha$ and $\rho=1$ (see \cite{hasebe2014}). Using  the series representation for the characteristic function of  $\mu^{\boxplus 1}(\d x)$ (see  \cite{hasebe2014}[Theorem 1.8]) coupled with the fact that  free stable distributions satisfy the scaling property $\mu^{\boxplus t}(\d x)=\mu^{\boxplus 1}(\d (xt^{-1/\alpha}))$, we compute
\begin{equation}\label{Laplace_free_stable}
\int_{[0,\infty)}e^{-zx} \mu^{\boxplus t}(\d x)=\sum\limits_{n\ge 0} \frac{(-1)^n z^{\alpha n} t^n}{n! \Gamma(2+(\alpha-1)n)}, \;\;\; z>0. 
\end{equation}
From Corollary \ref{corollary1} and the above result we derive the following formula for the cumulative distribution function of a positive stable random variable with parameter $\tilde \alpha=1-\alpha$: 
\begin{equation}\label{CDF_stable}
\nu^{*t}[0,w]=\sum\limits_{n\ge 0} \frac{(-1)^n t^n w^{-\tilde \alpha n}}{n! \Gamma(1-\tilde \alpha n)}, \;\;\; w>0.
\end{equation}
Formula \eqref{CDF_stable} is not new: it can be easily derived from the well-known series representation of the density of a stable distribution 
(see \cite{Zolotarev1986}[formula 2.4.8])
\begin{equation}\label{nu_series}
\nu^{*t}(\d y)=\frac{1}{\pi} \sum\limits_{n\ge 1} (-1)^{n-1} t^n \frac{\Gamma(1+\tilde \alpha n)}{n!} \sin(n \tilde \alpha \pi) y^{-\tilde \alpha n-1} \d y, \;\;\; y>0, 
\end{equation}
by writing $\nu^{*t}[0,w]=1-\nu^{*t}(w,\infty)$, integrating the infinite series \eqref{nu_series} term-by-term in order to compute $\nu^{*t}(w,\infty)$ and applying the reflection formula for the Gamma function and then simplifying the result. 

The computation in the above example could also be performed in the opposite direction: we could have started with the series representation \eqref{nu_series} and then derived the Laplace transform of a free stable distribution via Theorem \ref{theorem_main}, from where it would be easy to obtain the series representations for the density of free stable distributions. The latter formulas were originally derived in \cite{hasebe2014} using Mellin transform techniques, thus our Theorem \ref{theorem_main} provides an alternative method for obtaining such results. 

\subsection*{Example 2} 

Let $f(z)=\ln(1+z)$.  Then $f\in {\mathcal{CBF}}$, as can be easily verified by using Theorem \ref{theorem_CBF}(ii), and 
since $f(z)/z\to 0$ as $z\to +\infty$ we have $f\in {\mathcal{CBF}}^{\,\flat}$. It is well known that $f$ is the Laplace exponent of a Gamma process and the corresponding $*$-convolution semigroup is given by 
\begin{equation*}
\nu^{*t}(\d x)=\frac{1}{\Gamma(t)} x^{t-1} e^{-x} \d x, \;\;\; x>0. 
\end{equation*}
Thus $\nu^{*t}[0,y]=\gamma(t,y)/\Gamma(t)$, where $\gamma(t,y)$ is the incomplete Gamma function (see \cite{Jeffrey2007}[Section 8.35]). For $w,t>0$ we compute
\begin{align}\label{example1_eqn1}
\nonumber
\int_{[0,\infty)} e^{-wx} \mu^{\boxplus t}(\d x)&=
w^{-1} \int_0^w \nu^{*wt}[0,y]\d y=
\frac{w^{-1}}{\Gamma(wt)} \int_0^w \gamma(wt,y)\d y\\
\nonumber
&=\frac{w^{-1}}{\Gamma(wt)}
\Big[ y \gamma(wt,y) \Big \vert^w_0 - \int_0^{w} y \times y^{wt-1} e^{-y} \d y \Big]
\\ 
&= \frac{w^{-1}}{\Gamma(wt)}
\Big[ w \gamma(wt,w)-\gamma(wt+1,w)\Big]\\
\nonumber
&=
\frac{1}{\Gamma(wt)}
\big((1-t) \gamma(wt,w)+w^{wt-1} e^{-w} \big), 
\end{align}
where in the third step we used integration by parts and in the last step 
we used the identity 
$$
\gamma(a+1,x)=a\gamma(a,x)-x^a e^{-x}.
$$
Formula \eqref{example1_eqn1} implies a rather non-trivial result that the function
$$
w  \mapsto \frac{1}{\Gamma(wt)}
\big((1-t) \gamma(wt,w)+w^{wt-1} e^{-w} \big)
$$
is completely monotone, for every $t>0$. In particular, setting $t=1$ in \eqref{example1_eqn1} we obtain the following result
\begin{equation}
\frac{1}{\Gamma(w)}w^{w-1} e^{-w}=\int_{[0,\infty)} e^{-w x} \mu^{\boxplus 1}(\d x),\;\;\; w>0,
\end{equation}
which shows that the function $w \in (0,\infty) \mapsto w^{w-1} e^{-w}/\Gamma(w)$ is completely monotone. This fact was mentioned earlier in \cite{Burridge}[Example 6], however now we have identified that the measure appearing the Bernstein representation of this completely monotone function is free-regular and it corresponds to Voiculescu transform $\phi_{\mu^{\boxplus 1}}(z)=\ln(1-z)$. 

One could show that the inverse function of  $z\mapsto z+t \ln(1-z)$ is $F_{\mu^{\boxplus t}}(z)=1+W_{-1}(-e^{(z-1)/t}/t)$, where 
$W_{-1}$ is the non-principal real branch of the Lambert $W$-function (see  \cite{Burridge}[Proposition 4] and \cite{Corless96}). This gives us the Cauchy transform 
$G_{\mu^{\boxplus t}}$ and we could use Stieltjes inversion \eqref{Stieltjes_inversion} in order to recover the measure 
$\mu^{\boxplus t}(\d x)$ explicitly in terms of $W_{-1}$ function, however the resulting expression is complicated and does not seem to be particularly interesting or useful.

\subsection*{Example 3}

Next, we consider $f(z)=z/(z+1)$. This function belongs to ${\mathcal{CBF}}^{\,\flat}$, since it can be obtained from 
\eqref{CBF_integral_representation} by setting $a=1/2$, $b=0$ and $\rho(\d x)=\delta_1(\d x)$. 
We find 
$\phi_{\mu^{\boxplus t}}(z)=tz/(z-1)$ and $F^{-1}_{\mu^{\boxplus t}}(z)=z+tz/(z-1)$ for $z<0$. By finding the inverse function 
of $z \in (-\infty,0) \mapsto  z+tz/(z-1)$ and using the fact that $zF_t^{-1}(z) \to 1$ as $z\to -\infty$ we obtain 
$$
F_{\mu^{\boxplus t}}(z)=\frac{1}{2} \Big( z+1-t-\sqrt{(z+1-t)^2-4z}\Big), \;\;\; z<0. 
$$
Thus 
\begin{align*}
G_{\mu^{\boxplus t}}(z)=\frac{2}{z+1-t-\sqrt{(z+1-t)^2-4z}}
=\frac{1}{2z} \Big( z+1-t+\sqrt{(z+1-t)^2-4z}\Big),
\end{align*}
and using the Stieltjes inversion \eqref{Stieltjes_inversion} we conclude that 
\begin{equation}
\mu^{\boxplus t}(\d x)=\max(0,1-t) \delta_0(\d x)+\frac{1}{2\pi x} \sqrt{4t-(x-1-t)^2}\mathbf 1_{\{(1-\sqrt{t})^2<x<(1+\sqrt{t})^2\}} \d x,
\end{equation}
which is the well-known Marchenko-Pastur distribution (also known as the free Poisson distribution).

Having found explicitly the $\boxplus$-convolution semigroup $\mu^{\boxplus t}(\d x)$, our goal now is to do the same for the corresponding $*$-convolution semigroup $\nu^{*t}(\d y)$. We note that 
$$
f(z)=\frac{z}{z+1}=\int_0^{\infty} (1-e^{-zx})e^{-x} \d x, \;\;\; z>0,
$$
which means that the subordinator $Y$ is a compound Poisson process whose jumps are exponentially distributed. 
More precisely, we can write 
$$
Y_t=\sum_{n=1}^{N_t} \eta_i,
$$
where $N_t$ is the standard Poisson process and $\eta_i$ are independent random variables having distribution 
$\p(\eta_i \in \d x)=e^{-x} \d x$. From this representation we obtain
\begin{align*}
\nu^{*t}(\d y)&=\p(Y_t \in \d y)=\p(N_t=0) \delta_0(\d y)+\sum\limits_{n\ge 1} \p(N_t=n) \times \p(\eta_1+\eta_2+\dots+\eta_n \in \d y)\\
&=
e^{-t} \delta_0(\d y)+\sum\limits_{n\ge 1} \frac{t^ne^{-t}}{n!} \times \frac{y^{n-1} e^{-y}}{(n-1)!} \d y
=e^{-t} \delta_0(\d y)+\sqrt{t/y} e^{-t-y} I_1(2\sqrt{ty}) \d y,
\end{align*}
where $I_{\nu}(z)$ denotes the modified Bessel function of the second kind (see \cite{Jeffrey2007}).
Then formula \eqref{eqn_corollary_main} gives us the following result
\begin{align}
\nu^{*t}[0,w]&=e^{-t} + \int_0^w \sqrt{t/y} e^{-t-y} I_1(2\sqrt{ty}) \d y 
\\ \nonumber
&=\frac{\d}{\d w} \bigg[ \max(0,w-t)+\int_{(1-\sqrt{t/w})^2}^{(1+\sqrt{t/w})^2} e^{-xw} \frac{w}{2\pi x} \sqrt{4t/w-(x-1-t/w)^2}\d x \bigg],
\end{align} 
valid for $t,w>0$. We were not able to find this identity in \cite{Jeffrey2007} or other tables of integrals.

\subsection*{Example 4}

Here we take $f(z)=\sqrt{1+2z}$. Again, we check that $f \in {\mathcal{CBF}}^{\,\flat}$ using 
Theorem \ref{theorem_CBF}(ii) and checking that $f(z)/z\to 0$ as $z\to +\infty$. As in Example 3, we compute  
$$
F_{\mu^{\boxplus t}}(z)=z-t^2-t\sqrt{1+t^2-2z}, \;\;\; z<0, 
$$
and 
$$
G_{\mu^{\boxplus t}}(z)=\frac{z-t^2+t\sqrt{1+t^2-2z}}{z^2-t^2}, \;\;\; z<0.  
$$
The Stieltjes inversion \eqref{Stieltjes_inversion} gives us the $\boxplus$-convolution semigroup
$$
\mu^{\boxplus t}(\d x)=\max(0,1-t)\delta_t(\d x)+ \frac{t\sqrt{2x-1-t^2}}{\pi (x^2-t^2)} {\mathbf 1}_{\{x>(1+t^2)/2\}} \d x. 
$$
At the same time, it is well-known that the corresponding $*$-convolution semigroup comes from the Inverse Gaussian subordinator, so that we have 
$$
\nu^{*t}(\d y)=\frac{t}{\sqrt{2\pi y^3}} \exp(-t-(y-t)^2/(2y))\d y. 
$$
From formula \eqref{eqn_corollary_main} we obtain the following integral identity
\begin{align}
\p(Y_t\le w)&= \int_0^w \frac{t}{\sqrt{2\pi y^3}} \exp(-t-(y-t)^2/(2y))\d y\\
\nonumber &=
\frac{\d}{\d w}  \bigg[  e^{-t} \max(0,w-t)+t\int_{(1+(t/w)^2)/2}^{\infty} e^{-xw} \frac{\sqrt{2x-1-(t/w)^2}}{\pi(x^2-(t/w)^2)} \d x \bigg],
\end{align}
which again seems to be a new and non-trivial result.

\paragraph{Acknowledgements}
 The author would like to thank an anonymous referee for
careful reading of the paper and for suggesting several improvements. 
The research was supported by the
Natural Sciences and Engineering Research Council of Canada. 


\begin{thebibliography}{10}

\bibitem{Arizmendi_2013}
O.~Arizmendi, T.~Hasebe, and N.~Sakuma.
\newblock On the law of free subordinators.
\newblock {\em ALEA, Lat. Am. J. Probab. Math. Stat.}, 10(1):271 -- 291, 2013.

\bibitem{Bercovici_1993}
H.~Bercovici and D.~Voiculescu.
\newblock Free convolution of measures with unbounded support.
\newblock {\em Indiana University Mathematics Journal}, 42(3):733--773, 1993.

\bibitem{Bertoin}
J.~Bertoin.
\newblock Subordinators: examples and applications.
\newblock In {\em Lectures on probability theory and statistics
  ({S}aint-{F}lour, 1997)}, volume 1717 of {\em Lecture Notes in Math.}, pages
  1--91. Springer, Berlin, 1999.

\bibitem{Biane_1998}
P.~Biane.
\newblock Processes with free increments.
\newblock {\em Math. Z.}, 227:143 -- 174, 1998.

\bibitem{ECP1038}
K.~Borovkov and Z.~Burq.
\newblock Kendall's identity for the first crossing time revisited.
\newblock {\em Electron. Commun. Probab.}, 6:91--94, 2001.

\bibitem{Burridge}
J.~Burridge, A.~Kuznetsov, M.~Kwa\'snicki, and A.~Kyprianou.
\newblock New families of subordinators with explicit transition probability
  semigroup.
\newblock {\em Stochastic Processes and their Applications}, 124(10):3480 --
  3495, 2014.

\bibitem{Corless96}
R.~M. Corless, G.~H. Gonnet, D.~E.~G. Hare, D.~J. Jeffrey, and D.~E. Knuth.
\newblock On the {L}ambert {W} function.
\newblock In {\em Advances in Computational {M}athematics}, pages 329--359,
  1996.

\bibitem{Jeffrey2007}
I.~S. Gradshteyn and I.~M. Ryzhik.
\newblock {\em Table of integrals, series, and products}.
\newblock Elsevier/Academic Press, Amsterdam, seventh edition, 2007.

\bibitem{hasebe2014}
T.~Hasebe and A.~Kuznetsov.
\newblock On free stable distributions.
\newblock {\em Electron. Commun. Probab.}, 19:12 pp., 2014.

\bibitem{Kendall}
D.~G. Kendall.
\newblock Some problems in the theory of dams.
\newblock {\em Journal of the Royal Statistical Society. Series B
  (Methodological)}, 19(2):207--233, 1957.

\bibitem{Kyprianou}
{\relax A.E}.~Kyprianou.
\newblock {\em Fluctuations of {L}\'evy Processes with Applications:
  Introductory Lectures. Second Edition}.
\newblock Springer, 2014.

\bibitem{Perez-Abreu2012}
V.~P{\'e}rez-Abreu and N.~Sakuma.
\newblock Free infinite divisibility of free multiplicative mixtures of the
  {W}igner distribution.
\newblock {\em Journal of Theoretical Probability}, 25(1):100--121, 2012.

\bibitem{Sakuma}
N.~Sakuma.
\newblock On free regular infinitely divisible distributions.
\newblock {\em RIMS Kokyuroku Bessatsu}, B27:115--122, 2011.

\bibitem{Bernstein_book}
R.~L. Schilling, R.~Song, and Z.~Vondracek.
\newblock {\em Bernstein functions: theory and applications}.
\newblock De Gruyter Studies in Mathematics, 37. De Gruyter, 2010.

\bibitem{Zolotarev1986}
V.~M. Zolotarev.
\newblock {\em One-dimensional stable distributions}, volume~65 of {\em
  Translations of Mathematical Monographs}.
\newblock American Mathematical Society, Providence, RI, 1986.

\end{thebibliography}

\end{document}